\documentclass[a4paper,preprint]{elsarticle}
\usepackage{amsmath}
\usepackage{amsthm}
\newtheorem{thm}{Theorem}

\begin{document}

\begin{frontmatter}

\title{Another Proof of Oscar Rojo's Theorems}
\author[EP,MiS]{Hao Chen}
\ead{hao.chen@polytechnique.edu}
\address[EP]{Ecole Polytechnique, Palaiseau, France}
\address[MiS]{Max Planck Institute for Mathematics in the Sciences, Leipzig, Germany}

\author[MiS]{J\"urgen Jost\corref{cor1}}
\ead{jost@mis.mpg.de}

\begin{abstract}
We present here another proof of Oscar Rojo's theorems about the spectrum of graph Laplacian on certain balanced trees, by taking advantage of the symmetry properties of the trees in question, and looking into the eigenfunctions of Laplacian.
\end{abstract}


\end{frontmatter}

\section{Introduction}

Oscar Rojo has proved, first for balanced binary trees \cite{Rojo2002}, then extended to balanced trees such that vertices at the same level $l$ are of the same degree $d(l)$ \cite{Rojo2005}, that 

\begin{thm}\label{goal}
The spectrum of the graph Laplacian on such trees, is the union of the eigenvalues of $T_j, 1\le j\le k$, where $T_k$ is a tridiagonal $k\times k$ matrix
\begin{equation}
T_k=
\begin{pmatrix}
 1 & \sqrt{d(2)-1} & 0 &\cdots & 0 \\
\sqrt{d(2)-1} & d(2) & \ddots &   & \vdots\\
0 & \ddots & \ddots & \ddots & 0\\
\vdots &   & \ddots & d(k-1) & \sqrt{d(k)}\\
0 & \cdots & 0 & \sqrt{d(k)} & d(k)\\
\end{pmatrix}
\end{equation}
and $T_j$ with $j<k$ is the $j\times j$ principal submatrix of $T_k$. The multiplicity of each eigenvalue of $T_j$ as an eigenvalue of the Laplacian matrix is at least the difference of population between level $j$ et level $j+1$. 
\end{thm}

Please note that the levels are numerated from the leaves' level in Oscar Rojo's papers, i.e. the leaves are at level 1 and the root at level k.

He then found similar results \cite{Rojo2006} for a tree obtained by identifying the roots of two balanced trees. We call all the trees studied in these papers ``symmetric trees'' for a reason we will explain later.

We will give here another proof of Oscar Rojo's theorems, by looking at the eigenfunctions of Laplacian, and using the result of B{\i}y{\i}koglu's study on sign graphs \cite{Biyikoglu2003}. By our approach, we can see that because of the symmetry properties of the trees in question, the eigenfunctions of Laplacian have a stratified structure, which will be very useful for the proof.

The main goal of this paper is to prove the theorem \ref{goal}. The case studied in \cite{Rojo2006} will be briefly discussed at the end to show that our approach can be easily extended.

\section{Preparation}

Before we start, we shall agree on the notations and the terminology, which, for our convenience, are not the same as in Oscar Rojo's papers. We shall also study the symmetry of the trees in question to justify the name we gave to them.

In a rooted tree, every edge links a parent to one of its child. The \emph{root} is the only vertex without parent, and the \emph{leaves} are the vertices without child. We denote $C_v$ the set of children of vertex $v$, and label the children by integers $1,\ldots,|C_v|$.

The \emph{level} of a vertex $v$ is defined by its distance to the root, denoted by $l_v$. Therefore, unlike in Oscar Rojo's papers, for a $k$-level tree, root will be at level 0, and leaves at level $k-1$. If there exists a path $(v_0,...,v_n)$ such that $v_0=p$, $v_n=q$ and $v_i$ is the child of $v_{i-1}$ for all $1\le i\le n$, we call $p$ \emph{ancestor} of $q$ and $q$ \emph{descendant} of $p$. A vertex $v$ and all its descendants form the maximal subtree rooted at $v$, denoted by $T_v$.

There is one and only one path from the root to any non-root vertex. A vertex can thus be uniquely identified by listing in order the labels of non-root vertices on this path. For example, a vertex with identity $\{3,2,4,1\}$ means that it's the 1\textsuperscript{st} child of the 4\textsuperscript{th} child of the 2\textsuperscript{nd} child of the 3\textsuperscript{rd} child of the root. We call this list of label ``identity'' of $v$ and denote $I_v$. The root's identity is defined as an empty list $\{ \}$. The length of $I_v$ equals its level $l_v$. We denote $I_v^i$ the $i$\textsuperscript{th} number in $I_v$.

If a vertex $v$ have $n$ children, we denote $S_v$ the symmetric group $S_n$ acting on $C_v$ by relabeling the children of $v$. The action of $S_v$ on the vertex set $V_T$ (or by abuse of language, on the tree $T$) is defined by its action on $C_v$. That is, if a vertex $u$ with label $i$ is a child of $v$, it will be relabeled by $\sigma\in S_v$ to $\sigma(i)$. Labels of other vertices are not changed. In the words of identity, $\sigma\in S_v$ changes $I^{l_v}_u$ to $\sigma(I^{l_v}_u)$ if $u\in T_v$, and does not touch other numbers in $I_u$ or identities of other vertices.

\begin{figure}[ht]
  \centering
  \includegraphics[scale=0.25]{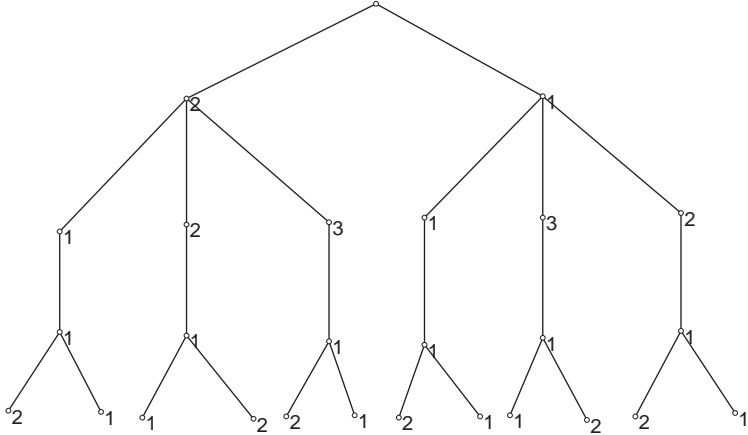}
  \caption{A typical symmetric tree. All the non-root vertices are labeled.}
\end{figure}

The trees studied in the papers of Oscar Rojo are such that vertices at the same level $l$ have the same number of children $c(l)$, i.e. are of the same degree $d(l)$ (see figure 1). Let the total number of vertices at level $l$ be $n(l)$, we have $n(0)=1$ and $c(l)=d(l)-1$, $n(l)=n(l-1)c(l-1)=\prod_{i=0}^{l-1}c(i)$ for $l>0$.

We desire to call them "symmetric trees" because they are invariant under the action of $S_v$ for all vertices $v$. Here ``invariant'' implies that the vertex set $V_T$ is invariant. Since vertices are identified by their identities, the invariance means that for any vertex, its identity can be found in the vertex set of the tree after the action of $S_v$. In this sense, the action of $S_v$ on a symmetric tree is actually a permutation of the subtrees rooted at the children of $v$. See figure 2 for a tree not invariant under the action of $S_v$.

\begin{figure}[ht]
  \centering
  \includegraphics[scale=0.25]{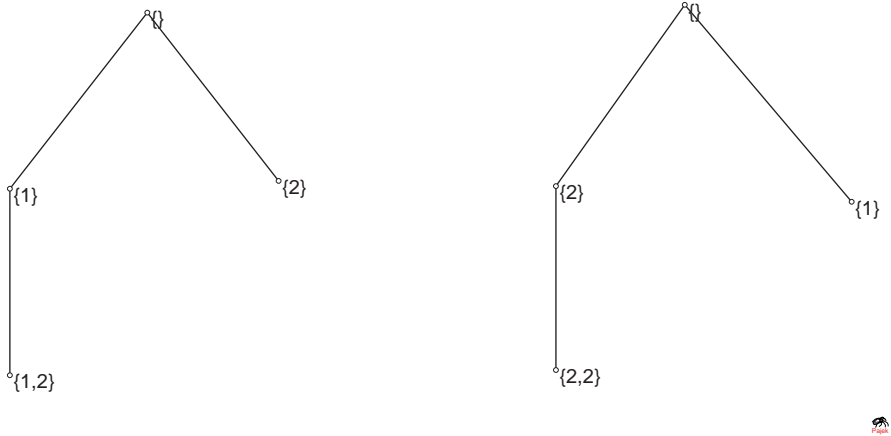}
  \caption{An example showing a non-symmetric tree which is not invariant under the action of $S_{\text{root}}$, the identity $\{1,2\}$ disappears and $\{2,2\}$ appears after relabeling the children of root.}
\end{figure}

Let $\mathcal{F}$ be the set of real-valued functions on the set of vertices $V_T$ (or by abuse of language, on the tree $T$). If $T$ is symmetric, because of the invariance, we can define the action of $S_v$ on $\mathcal{F}$ by $(\sigma\circ f)(T)=f(\sigma\circ T)$, where $f\in\mathcal{F}$ and $\sigma\in S_v$.

\section{Sign Graph of Trees}
Our proof benifits from the studies on sign graphs. The sign graph (strong discrete nodal domain) is a discrete version of Courant's nodal domain. 
Consider $G=(V,E)$ and a real-valued function $f$ on $V$. A \emph{positive (resp. negative) sign graph} is a maximal, connected subgraph of $G$ with vertex set $V'$, such that $f|_{V'}>0$(resp. $f|_{V'}<0$).

The study of sign graphs often deals with generalized Laplacian. A matrix $M$ is called generalized Laplacian matrix of graph $G=(V,E)$ if $M$ has nonpositive off-diagonal entries, and $M(u,v)<0$ if and only if there is an edge in $E$ connecting $u$ and $v$. Obviously, the Laplacian $\mathcal{L}$ studied in Oscar Rojo's papers
\begin{equation}
\mathcal{L}(u,v)=\begin{cases}
\deg{u}& \text{if}\ u=v\\  
-1 & \text{if}\ (u,v)\in E \\
0 & \text{otherwise}
\end{cases}
\end{equation}
is a generalized Laplacian. 

A \emph{Dirichlet Laplacian} $\mathcal{L}_{\Omega}$ on a vertex set $\Omega$, is an operator defined on $\mathcal{F}_{\Omega}$, the set of real-valued functions on $\Omega$. It is defined by
\begin{equation}
\mathcal{L}_{\Omega}f=(\mathcal{L}\tilde{f})|_{\Omega}
\end{equation}
where $\tilde{f}\in\mathcal{F}$ vanishes on $V-\Omega$ and equals to $f$ on $\Omega$. It can be regarded as a Laplacian defined on a subgraph with limit condition, and has many properties similar to those of Laplacian. A Dirichlet Laplacian is also a generalized Laplacian.

\cite{Gladwell2001,Biyikoglu2004,Biyikoglu2007} have proved the following discrete analogues of Courant's Nodal Domain Theorem for generalized Laplacian:

\begin{thm}
Let $G$ be a connected graph and $A$ be generalized Laplacian of $G$, let the eigenvalues of $A$ be non-decreasingly ordered, and $\lambda_k$ be a eigenvalue of multiplicity $r$, i.e.
\begin{equation}
\lambda_1\le\cdots\le\lambda_{k-1}<\lambda_k=\cdots=\lambda_{k+r-1}<\lambda_{k+r}\le\cdots\le\lambda_n
\end{equation}
Then a $\lambda_k$-eigenvalue has at most $k+r-1$ sign graphs.
\end{thm}

And, \cite{Biyikoglu2003} has studied the nodal domain theory of trees. Instead of inequality, we have kind of equality on trees. But we have to study two cases
 
\begin{thm}[T\"urker Biyikoglu]\label{Turker1}
Let $T$ be a tree, let $A$ be a generalized Laplacian of $T$. If $f$ is a $\lambda_k$-eigenfunction without a vanishing coordinate (vertex where $f=0$), then $\lambda_k$ is simple and $f$ has exactly $k$ sign graphs.
\end{thm}

\begin{thm}[T\"urker Biyikoglu]\label{Turker2}
Let $T$ be a tree, let $A$ be a generalized Laplacian of $T$. Let $\lambda$ be an eigenvalue of $A$ all of whose eigenfunctions have vanishing coordinates. Then
\begin{enumerate}
\item eigenfunctions of $\lambda$ have at least one common vanishing coordinates. \item Let $Z$ be the set of all common vanishing points, $G-Z$ is then a forest with component $T_1,\ldots,T_m$, Let $A_1,\ldots,A_m$ be restriction of $A$ to $T_1,\ldots,T_m$, then $\lambda$ is a simple eigenvalue of $A_1,\ldots,A_m$, and $A_i$ has a $\lambda$-eigenfunction without vanishing coordinates, for
$i=1,\ldots,m$. 
\item Let $k_1,\ldots,k_m$ be the positions of $\lambda$ in the spectra of $A_1,\ldots,A_m$ in non decreasing order. Then the number of sign graphs of an eigenfunction of $\lambda$ is at most $k_1+\ldots+k_m$, and there exists an $\lambda$-eigenfunction with $k_1+\ldots+k_m$ sign graphs.
\end{enumerate}
\end{thm}

In this theorem, if $A$ is the Laplacian, $A_i$ in the second point are in fact the Dirichlet Laplacians on $T_i$. 

\section{Proof}
We first prove the following theorem.

\begin{thm}
Consider a symmetric tree $T$ rooted at $r$. If $f$ is a $\lambda$-eigenfunction of $\mathcal{L}$ such that $f(r)\ne 0$, then we can find a $\lambda$-eigenfunction $\tilde{f}$ such that $l_u=l_v\Rightarrow \tilde{f}(u)=\tilde{f}(v)$, i.e. $\tilde{f}$ takes a same value on vertices at the same level. We call $\tilde{f}$ \emph{the stratified eigenfunction} (see figure 3).
\end{thm}

\begin{figure}[ht]
  \centering
  \includegraphics[scale=0.25]{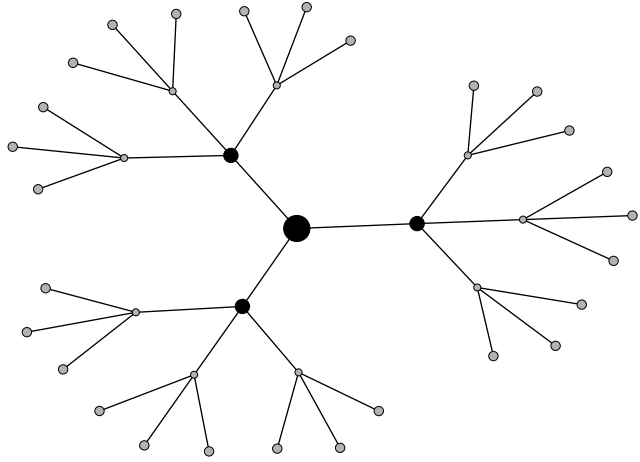}
  \caption{An example of stratified structure on the whole tree. The size of vertex indicates the absolute value of $\tilde{f}$, the color of vertex indicates the sign, gray for negative, black for positive, white for 0.}
\end{figure}

\begin{proof}

We first consider the case where $f$ has no vanishing coordinate, and prove that $f$ is itself a stratified eigenfunction. By the theorem \ref{Turker1} we know that $\lambda$ is simple. 

Assume that $l_u=l_v$ and $f(u)\ne f(v)$. There is a unique path connecting $u$ and $v$. Choose on this path a vertex $p$ with the lowest level (closest to the root), it is the first common ancestor of $u$ and $v$, so $I_u^{l_p+1}$ and $I_v^{l_p+1}$ are the labels of the last different ancestors of $u$ and $v$. Take $\sigma=\tau(I_u^{l_p+1}, I_v^{l_p+1})\in S_p$ where $\tau(i,j)$ is the transposition of $i$ and $j$, $\sigma\circ f$ is then another $\lambda$-eigenfunction by symmetry, but $\sigma\circ f\ne f$ since $f(u)\ne f(v)$, which violates the simplicity of $\lambda$. So $f(u)=f(v)$ and $f$ itself is a the stratified eigenfunction. This is also true for Dirichlet Laplacian.

We study then the case with vanishing points. By theorem \ref{Turker2}, we know that all $\lambda$-eigenfunctions have at least one common vanishing coordinate. Let $v$ be a common vanishing coordinate, a vertex $u$ of the same level must also be one. If it is not the case, a same argument as above will lead to a contradiction that there is another $\lambda$-eigenfunction who does not vanish on $v$. We can now use the term ``vanishing level''. There can be no two consecutive vanishing levels, if it happens, all the lower levels must vanish, so is the root, which is assumed to be non zero.

The common vanishing coordinates divide the tree into components without vanishing points. By theorem \ref{Turker2}, $\lambda$ is a simple eigenvalue of Dirichlet Laplacian on each component. Components between two vanishing levels are identical, so they have a same eigenfunction (up to a factor) for Dirichelt Laplacian. 

We can construct $\tilde{f}$ as following: Choose a vanishing vertex $v$ on whose children $f$ takes different values. Symmetrize the children of $v$ by taking 
\begin{equation}
f'=\frac{1}{|C_v|!}\sum_{\sigma\in S_v}\sigma\circ f
\end{equation}
If $f'$ is a stratified eigenfunction, $\tilde{f}=f'$ and the construction is finished. If not, there must be another vertex whose children are not symmetrized, then let $f=f'$ and repeat the procedure.

This procedure will end in a finite number of steps because the symmetrization can not be undone and the tree is finite. When all the vertices have their children symmetrized, the result must be stratified, because $\lambda$ is simple as an eigenvalue of Dirichlet Laplacian on each non-vanishing component.
\end{proof}

This theorem is also true for Dirichlet Laplacian on a rooted subtree (see figure 4). Together with the following theorem, we can find all the eigenfunctions of Laplacian.

\begin{figure}[ht]
  \centering
  \includegraphics[scale=0.25]{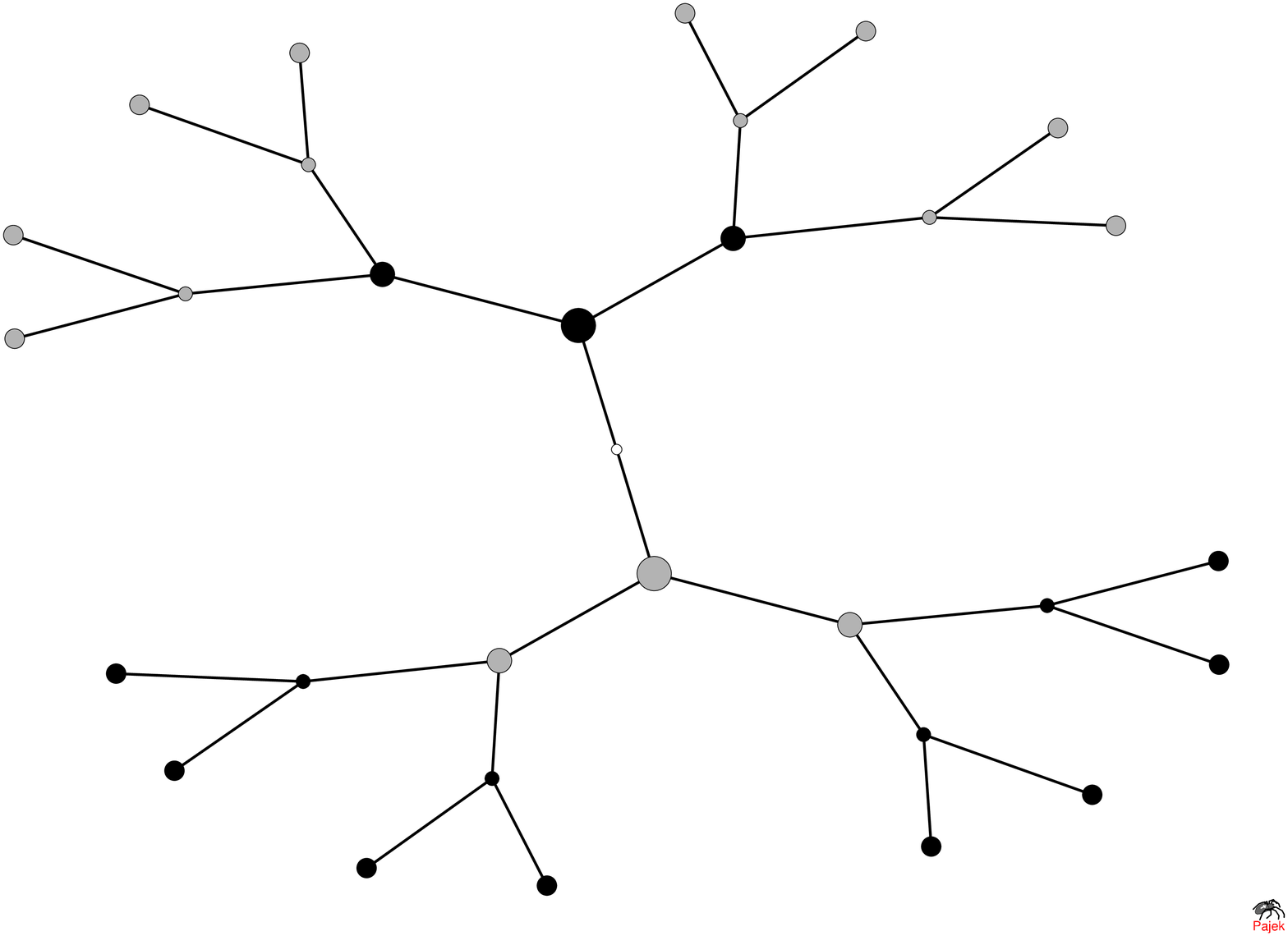}
  \caption{An example of stratified structure on subtrees.}
\end{figure}

\begin{thm}
Consider a symmetric tree $T$ with root $r$. Let $v$ be a non-root vertex. If $\lambda$ is an eigenvalue of $\mathcal{L}_{T_v}$ whose eigenfunctions do not vanish at $v$, then $\lambda$ is also an eigenvalue of $\mathcal{L}$, whose multiplicity is at least $n(l_v)-n(l_v-1)$.
\end{thm}
\begin{proof}
Let $u$ be a non-root vertex of a symmetric tree $T$ and $p$ its parent. We first notice that, if $\lambda$ is an eigenvalue of Dirichlet Laplacian $\mathcal{L}|_{T_u}$, $\lambda$ is also an eigenvalue of Laplacian on $\mathcal{L}|_{T_p}$. Because for a $\lambda$-eigenfunction $f$ of Dirichlet Laplacian $\mathcal{L}|_{T_u}$, let $i$ be the label of $u$, $\forall j\ne i$, the action of $\tau(i,j)\in S_p$ on $f$ will help constructing another $\lambda$-eigenfunction $f'$ as, 
\begin{equation}\label{copy}
f'=\begin{cases}
f-\tau(i,j)\circ f & \text{if}\ f(u)\ne 0\\
\tau(i,j)\circ f & \text{if}\ f(u)=0
\end{cases}
\end{equation}

We see that the multiplicity of $\lambda$ as an eigenvalue of $\mathcal{L}|_{T_p}$ is $|C_p|-1$ if $f(u)\ne 0$, because $f'$ constructed by differenct $j$ are independent. If $f(u)=0$, $f$ itself is also an eigenfunction of $\mathcal{L}|_{T_p}$ so the multiplicity is $|C_p|$. 

If $f(v)\ne 0$, let $p$ be the parent of $v$. We conclude by a recursive argument, that the multiplicity of $\lambda$ as an eigenvalue of $\mathcal{L}=\mathcal{L}|_{T_r}$ is 
\begin{equation}
(|C_p|-1)\prod_{i=0}^{l_p-1}c(i)=\prod_{i=0}^{l_p}c(i)-\prod_{i=0}^{l_p-1}c(i)=n(l_p+1)-n(l_p)=n(l_v)-n(l_v-1)
\end{equation}
\end{proof}

It is obvious that a stratified eigenfunction of (Dirichlet) Laplacian can not vanish on the root of (sub)tree, otherwise it will vanish everywhere. So for a vertex $v$ of a $k$-level symmetric tree, we can find $k-l_v$ stratified independent eigenfunctions of $\mathcal{L}|_{T_v}$ (because $T_v$ have $k-l_v$ levels), and none of them vanishes at $v$. The eigenfunctions who vanishes at $v$ are independent to these stratified eigenfunctions, because of the ``-'' in the $f'$ given in equation \ref{copy}.

Let's count how many independent eigenfunctions have been found. For a vertex $v$, we find $k-l_v$ stratified eigenfunction of $\mathcal{L}|_{T_v}$. For each of them, there are $n(l_v)-n(l_v-1)$ independent eigenfunctions of $\mathcal{L}$ with the same eigenvalue, who are in fact stratified on maximal subtrees rooted on vertices at level $l_v$. As a special case, the root have $k$ stratified eigenfunctions. The total number of independent eigenfunctions is then (by summation by parts)
\begin{equation}
\sum_{i=1}^{k-1}(k-i)[n(i)-n(i-1)]+k=-(k-1)n(0)+\sum_{i=1}^kn(i)+k=\sum_{i=0}^k n(i)=|V_T|
\end{equation}
Which means that all the eigenfunctions have been found.

Now we can prove Oscar Rojo's theorem \ref{goal}

\begin{proof}[Proof of Theorem \ref{goal}]
We search for stratified eigenfunctions of the (Dirichlet) Laplacians on symmetric (sub)trees. In fact, knowing that the eigenfunctions are stratified, we can regard them as a function of level, and write the Laplacian equation for stratified eigenfunctions as $\mathcal{L}f(l)=d(l)f(l)-f(l-1)-c(l)f(l+1)=\lambda f$. For a subtree rooted at level $l_0$, $\lambda$ is an eigenvalue of matrix
\begin{equation}
S=
\begin{pmatrix}
 d(l_0) & -c(l_0) & 0 &\cdots & 0 \\
 -1 & d(l_0+1) & \ddots &   & \vdots\\
0 & \ddots & \ddots & \ddots & 0\\
\vdots &   & \ddots & d(k-2) & -c(k-2)\\
0 & \cdots & 0 & -1 & d(k-1)\\
\end{pmatrix}
\end{equation}
Let $R$ be a $(k-l_0)\times(k-l_0)$ diagonal matrix whose i\textsuperscript{th} element is $(-1)^i\sqrt{n(i-1+l_0)},1\le i\le k-l_0$, we have
\begin{equation}
T=RSR^{-1}=
\begin{pmatrix}
 d(l_0) & \sqrt{c(l_0)} & 0 &\cdots & 0 \\
 \sqrt{c(l_0)} & d(l_0+1) & \ddots &   & \vdots\\
0 & \ddots & \ddots & \ddots & 0\\
\vdots &   & \ddots & d(k-2) & \sqrt{c(k-2)}\\
0 & \cdots & 0 & \sqrt{c(k-2)} & d(k-1)\\
\end{pmatrix}
\end{equation} 
which is similar to $S$. Note that $d(k-1)=1$, the matrix given by Oscar Rojo is recognized (in a different notation) when $l_0=0$, as well as its principal submatrices when $l_0>0$.
\end{proof}

\section{Combinition of two symmetric trees}
The tree $T$ studied in \cite{Rojo2006} are obtained by identifying the roots of two symmetric trees $T_1$ and $T_2$. We can still use the terminology ``level'' by allowing negative levels, and use the sign of level to distinguish the two symmetric subtrees. So a vertex has always a bigger absolute level than its parent. The tree is still invariant under $S_v$, except for the root, where we must distinguish the children at $+1$ level and at $-1$ level, and define respectively the action of $S_r^+$ and $S_r^-$. The stratified structure is still valid.

Let's count again how many eigenfunctions can be found with the help of stratified structure. An eigenfunction $f$ of $\mathcal{L}|_{T_1}$ who vanishes at the root is also an eigenfunction of $T$. From the previous analysis, we can find $|V_{T_1}|-k_1$ such eigenfunctions, where $k_1$ is the number of levels in $T_1$. Same for $T_2$. Then we search for the stratified eigenfunctions of $\mathcal{L}$, there are $k_1+k_2-1$. We have found in all $|V_{T_1}|+|V_{T_2}|-1=|V_T|$ independent eigenfunctions. That is all the eigenfunctions.

So by the same argument as before, with the knowledge of stratified structure, we can write the Laplacian equation for stratified eigenfunctions, and the result of Oscar Rojo is immediate.

\section*{Acknowledgment}
This work was done at Max-Planck-Institut f\"ur Mathematik in den Naturwissenschaften as an internship at the end of my third year in Ecole Polytechnique. I'd like to thank Frank Bauer for the very helpful discussions and suggestions. The figures in this paper are generated by Pajek, a program for analysing graphs.


\bibliographystyle{model1-num-names}
\bibliography{References}

\end{document}